\documentclass[10pt]{amsart}

\usepackage{mathrsfs}
\usepackage{amstext}
\usepackage{amscd}
\usepackage{latexsym}
\usepackage{amsfonts}
\usepackage{amssymb}
\usepackage{amsmath}
\usepackage{amsthm}
\usepackage{color}
\usepackage{multicol}
\usepackage[pdftex]{graphicx}
\graphicspath{{./pict/}}            
\DeclareGraphicsExtensions{.pdf,.png,.jpg}

\usepackage[cp1251]{inputenc}
\usepackage[russian,english]{babel}
\theoremstyle{plain}
\newtheorem{theorem}{Theorem}[section]
\newtheorem{lemma}[theorem]{Lemma}
\newtheorem{prepos}[theorem]{Proposition}

\theoremstyle{definition}
\newtheorem{definition}[theorem]{Definition}

\newtheorem{remark}[theorem]{Remark}
\newtheorem{example}[theorem]{Example}

\renewcommand{\Im}{\operatorname{Im}}
\renewcommand{\Re}{\operatorname{Re}}

\begin{document}

\author[M.~Tyaglov]{Mikhail Tyaglov}

\address{School of Mathematical Sciences, Shanghai Jiao Tong University\\
and  Faculty of Mathematics, Far East Federal University}
\email{tyaglov@sjtu.edu.cn}

\title[Self-interlacing polynomials]{Self-interlacing polynomials}












\begin{abstract}
We describe a new subclass of the class of real polynomials with real simple roots called self-interlacing polynomials.
This subclass is isomorphic to the class of real Hurwitz stable polynomials (all roots in the open
left half-plane). In the work, we present basic properties of self-interlacing polynomials and their relations with Hurwitz and Hankel
matrices as well as with Stiltjes type of continued fractions. We also establish ``self-interlacing'' analogues
of the well-known Hurwitz and Li\'enard-Chipart criterions for stable polynomials. A criterion
of Hurwitz stability of polynomials in terms of minors of certain Hankel matrices is established.
\end{abstract}

\maketitle

%

%


\setcounter{equation}{0}

\section{Introduction}

A real polynomial $p(z)$ is called \emph{self-interlacing} if all its roots are real and simple
(of multiplicity one) and interlace the roots of the polynomial $p(-z)$.

It is easy to see that some numbers $\lambda_j$, $j=1,\ldots,n$, being ordered as decreasing their absolute values,
are the roots of the polynomial $p(z)$ if and only if they satisfy one of the following
inequalities
\begin{equation}\label{roots.type.I}
\lambda_1>-\lambda_2>\lambda_3>\cdots>(-1)^{n-1}\lambda_n>0,
\end{equation}
or
\begin{equation}\label{roots.type.II}
-\lambda_1>\lambda_2>-\lambda_3>\cdots>(-1)^n\lambda_n>0.
\end{equation}

If the roots of the polynomial $p(z)$ are distributed as in~\eqref{roots.type.I}, then $p(z)$
is called self-interlacing of kind~$I$. Respectively, if the roots of $p(z)$ are
distributed as in~\eqref{roots.type.II}, then $p(z)$
is called self-interlacing polynomial of kind $II$. We denote the class of all self-interlacing polynomials
as~$\mathbf{SI}$. Respectively, the classes of all self-interlacing polynomials of kind $I$ and $II$ are denoted
as $\mathbf{SI}_{I}$ and~$\mathbf{SI}_{II}$.

\vspace{2mm}

The self-interlacing polynomials seem to implicitly appear first time in~\cite[Lemma~2.6]{Fisk} (in its first edition) where a necessary condition for
real polynomials to be self-interlacing was established (see Theorem~\ref{Th.Stodola.necessary.condition.self-interlacing} of the present work). In~\cite{Holtz_SI} (see
also~\cite{Tyaglov_Schwarz} and references there)
there were introduced tridiagonal matrices of the form:
\begin{equation}\label{main.matrix.intro}
\begin{pmatrix}
    b_1 & b_2 &  0 &\dots&   0   & 0 \\
    b_2 & 0 &b_3 &\dots&   0   & 0 \\
     0  &b_3 & 0 &\dots&   0   & 0 \\
    \vdots&\vdots&\vdots&\ddots&\vdots&\vdots\\
     0  &  0  &  0  &\dots&0& b_n\\
     0  &  0  &  0  &\dots&b_{n}&0\\
\end{pmatrix},\qquad b_1\neq0,\quad b_k>0,\quad k=2,\ldots,n,
\end{equation}
whose characteristic polynomials belong to $\mathbf{SI}_{I}$ if $b_1>0$ (it was established implicitly without mentioning of self-interlacing polynomials).
In~\cite{Holtz_SI}, the author also proved that for any polynomial $p\in\mathbf{SI}_{I}$, there exists a unique matrix of the form~\eqref{main.matrix.intro} with $b_1>0$ whose characteristic polynomial is~$p$. Clearly, for $b_1<0$, we deal with the polynomials in the class $\mathbf{SI}_{II}$.

It is easy to see that $p(z)\in\mathbf{SI}_{I}$ if
and only if $p(-z)\in\mathbf{SI}_{II}$, so it is sufficient to study only one of the classes, say, $\mathbf{SI}_{I}$.

Note that self-interlacing polynomials appear in the theory of orthogonal polynomials. For example,
the Chebyshev polynomials of III and IV kinds are self-interlacing, see e.g.~\cite{MasonHandscomb}.
In fact, there are many curious and useful self-interlacing
 orthogonal polynomials, and we plan to devote a separate research to this interesting topic.
Self-interlacing polynomials also appear as characteristic polynomials of certain structured matrices,
see~\cite{Holtz_SI,Tyaglov_Schwarz,Tyaglov_SI_2}.

One of the main results of the present work is the following theorem.
\begin{theorem}\label{Theorem.isomorphism}
The class $\mathbf{SI}_{I}$ is isomorphic to the class of real Hurwitz stable polynomials.
\end{theorem}
Recall that a real polynomial is called Hurwitz stable (or stable) if its zeroes lie in the \textit{open}
left half-plane of the complex plane. The mentioned isomorphism is a linear operator acting on the coefficients of polynomials. Precisely,
a polynomial
\begin{equation}\label{main.poly}
p(z)=a_0z^n+a_1z^{n-1}+a_2z^{n-22}+a_3z^{n-3}+a_4z^{n-4}+\cdots+a_n
\end{equation}
belongs to $\mathbf{SI}_I$ if and only if the polynomial
\begin{equation}\label{main.poly.q}
q(z)=a_0z^n-a_1z^{n-1}-a_2z^{n-2}+a_3z^{n-3}+a_4z^{n-4}-\cdots+(-1)^{\tfrac{n(n+1)}{2}}a_n
\end{equation}
is Hurwitz stable (see Theorem~\ref{Theorem.connection.Hurwitz.self-interlacing}).

%
%

Due to the isomorphism 
self-interlacing polynomials have some properties similar to ones of Hurwitz stable polynomials.
In this work, we prove the following analogue of Hurwitz criterion.

\begin{theorem}[analogue of Hurwitz's criterion]\label{Theorem.self-interlacing.Hurwitz.criterion}
A real polynomial $p$ belongs to the class~$\mathbf{SI}_I$ if and only if
its Hurwitz minors $\Delta_k(p)$ satisfy the inequalities
\begin{equation*}
(-1)^j\Delta_{2j-1}(p)>0,\qquad j=1,\ldots,\left[\dfrac {n+1}2\right],
\end{equation*}
\begin{equation*}
\Delta_{2j}(p)>0,\qquad j=1,\ldots,\left[\dfrac {n}2\right],
\end{equation*}
where $[\rho]$ denotes the largest integer not exceeding $\rho$.
\end{theorem}

Recall that for a real polynomial $p$ defined in~\eqref{main.poly}, its \textit{Hurwitz minors} $\Delta_k(p)$ have the form
\begin{equation}\label{Hurwitz.minors}
\Delta_{k}(p)=
\begin{vmatrix}
a_1&a_3&a_5&a_7&\dots&a_{2k-1}\\
a_0&a_2&a_4&a_6&\dots&a_{2k-2}\\
0  &a_1&a_3&a_5&\dots&a_{2k-3}\\
0  &a_0&a_2&a_4&\dots&a_{2k-4}\\
\vdots&\vdots&\vdots&\vdots&\ddots&\vdots\\
0  &0  &0  &0  &\dots&a_{k}
\end{vmatrix},\quad k=1,\ldots,n,
\end{equation}
where we set $a_i\equiv0$ for $i>n$.

Another criterion for stability of real polynomials whose analogue also can be established
for self-interlacing polynomials is the Li\'enard--Chipart criterion~\cite{LienardChipart} (see also~\cite{Gantmakher,Holtz_Tyaglov}).
\begin{theorem}[analogue of Li\'enard--Chipart's criterion]\label{Theorem.Lienard.Chipart.intro}
A real polynomial $p$ belongs to the class~$\mathbf{SI}_I$ if and only if
the following inequalities hold
\begin{equation*}
(-1)^j\Delta_{2j-1}(p)>0,\qquad j=1,\ldots,\left[\dfrac {n+1}2\right],
\end{equation*}
\begin{equation*}
(-1)^ja_{2j}>0,\qquad j=1,\ldots,\left[\dfrac {n}2\right].
\end{equation*}
\end{theorem}

\vspace{2mm}

As well, we provide one more criterion for polynomials to be self-interlacing (Theorem~\ref{Theorem.SI.first.new}) which together with Theorem~\ref{Theorem.isomorphism}
and formul\ae~\eqref{main.poly}--\eqref{main.poly.q} provides a seemingly new criterion of stability of real polynomials.
Namely, a polynomial $p(z)$ of degree $n$ is Hurwitz stable if and only if the following inequalities hold (Theorem~\ref{Theorem.new.stability.criterion})
$$
(-1)^{\tfrac{j(j+1)}2}D_j(R)>0, \qquad j=1,\ldots,n,
$$
where $D_j(R)$ are the Hankel minors
$$
D_j(R)=\begin{vmatrix}
    s_0 &s_1 &s_2 &\dots &s_{j-1}\\
    s_1 &s_2 &s_3 &\dots &s_j\\
    \vdots&\vdots&\vdots&\ddots&\vdots\\
    s_{j-1} &s_j &s_{j+1} &\dots &s_{2j-2}
\end{vmatrix},\quad j=1,2,3,\dots,
$$
constructed with the coefficient of the Laurent series of the function
$$
R(z):=\dfrac{(-1)^np(-z)}{p(z)}=1+\frac{s_0}z+\frac{s_1}{z^2}+\frac{s_2}{z^3}+\dots,
$$
It is interesting to note that the function $R(z)$ related to a Hurwitz stable polynomial maps
the open right half-plane of the complex plane to the open unit disc (Theorem~\ref{Theorem.new.stability.criterion}),
while $R(z)$ related to a self-interlacing polynomial of kind $I$ maps the upper half-plane
to the lower half-plane (Theorem~\ref{Theorem.SI.first.new}).

Finally, we would like to inform the reader that instead of citing various papers
regarding rational functions, matrices and especially rational $R$-functions, we cite
the survey~\cite{Holtz_Tyaglov} that was written exactly for such quotations. All other
references can be found there. Some results of this work were mentioned in the technical
report~\cite{Tyaglov_GHP}, but here we substantially simplified most of the proofs
and provided additional properties of self-interlacing polynomials.

The parer is organized as follows. In Section~\ref{section:Basic.theorems}
we introduce the basic objects and prove basic Theorems~\ref{Theorem.SI.first.new}--\ref{Theorem.main.self-interlacing} which
imply immediately Theorems~\ref{Theorem.self-interlacing.Hurwitz.criterion}--\ref{Theorem.Lienard.Chipart.intro} due to properties
of Hurwitz minors and rational $R$-functions (see e.g.~\cite{Holtz_Tyaglov}).
Section~\ref{section:connection.Hurwitz.SI}
is devoted to Theorem~\ref{Theorem.isomorphism} and its consequences whose proofs are also based on
Theorems~\ref{Theorem.SI.first.new}--\ref{Theorem.main.self-interlacing} and properties of $R$-functions. In Section~\ref{section:properties}
we study some properties of self-interlacing polynomials. In particular, we study
the minors of Hurwitz matrix of self-interlacing polynomials. In Section~\ref{section:Second.proof.of.SI.criterion},
we establish some curious determinant formul\ae\ and prove the stability criterion mentioned above.


\setcounter{equation}{0}

\section{Hurwitz and Li\'enard--Chipart criterions}\label{section:Basic.theorems}

Consider a real polynomial
\begin{equation}\label{main.polynomial}
p(z)=a_0z^n+a_1z^{n-1}+\dots+a_n,\qquad a_1,\dots,a_n\in\mathbb
R,\ a_0>0.
\end{equation}
In the rest of the paper, we use the following notation
\begin{equation}\label{floor.poly.degree}
l\stackrel{def}{=}\left[\dfrac {n+1}2\right],
\end{equation}
where $n=\deg p$, and $[\rho]$ denotes the largest integer not
exceeding $\rho$.

We introduce the following two auxiliary rational functions associated with the polynomial $p(z)$:

\begin{equation}\label{assoc.function}
z\Phi(z^2)=\displaystyle\frac {p(z)-(-1)^np(-z)}{p(z)+(-1)^np(-z)}=\dfrac{a_{1}z^{n-1}+a_{3}z^{n-3}+a_{5}z^{n-5}+\cdots}{a_0z^n+a_{2}z^{n-2}+a_{4}z^{n-4}+\cdots}\ ,
\end{equation}
and
\begin{equation}\label{auxiliary.rational.function}
R(z)=\dfrac{(-1)^np(-z)}{p(z)}\ .
\end{equation}
Note that for odd $n$ the function $\Phi(u)$ has a pole at zero.

It is easy to see that these functions are related as follows
\begin{equation}\label{poly1.1}
z\Phi(z^2)=\displaystyle\frac{1-R(z)}{1+R(z)},
\end{equation}
and
\begin{equation}\label{poly1.3}
R(z)=\displaystyle\frac{1-z\Phi(z^2)}{1+z\Phi(z^2)}\,.
\end{equation}

It turns out that for the polynomial $p(z)$ being self-interlacing, the functions $\Phi(u)$ and $R(z)$ possess some remarkable properties.
Namely, the following two theorems hold.
\begin{theorem}\label{Theorem.SI.first.new}
Let $p(z)$ be a real polynomial. The polynomial $p(z)\in\mathbf{SI}_{I}$ if and only if the function $R(z)$ defined in~\eqref{auxiliary.rational.function}
maps the upper half-plane of the complex plane to the lower half-plane and has exactly $n$ poles.
\end{theorem}
\begin{proof}
Let the polynomial $p(z)$ be self-interlacing of kind $I$. Then by definition, the roots and poles of the function $R(z)$ are real, simple and interlacing.
In particular, $R(z)$ has exactly $n$ poles.
Thus $R(z)$ maps the upper half-plane to itself or to the lower half-plane, and it is monotone
on the real line between its poles  (see e.g.~\cite[Chapter 3]{Holtz_Tyaglov}). Now since $R(z)\to1$ as $z\to+\infty$, and its largest zero $-\lambda_{2}(>0)$ is smaller than
its largest pole $\lambda_1(>0)$, where $\lambda_j$ are the roots of $p(z)$, one has $R(z)>1$ on the interval $(\lambda_1,+\infty)$.
Consequently, $R(z)$ is decreasing on the real line and therefore maps the upper half-plane to the lower half-plane (see~\cite{Holtz_Tyaglov}), as required.

Conversely, if $R(z)$ maps the upper half-plane to the lower half-plane and has exactly $n$ poles, then the roots of its numerator $p(-z)$ and denominator $p(z)$
are real, simple and interlacing (see~\cite[Theorem 3.4]{Holtz_Tyaglov}), and $p(z)$ and $p(-z)$ have no common roots. Moreover, since $R(z)$ is decreasing on the real line, its largest
pole (the largest root of $p(z)$) is greater than its largest root (the largest root of $p(-z)$), therefore, $p(z)\in\mathbf{SI}_{I}$.
\end{proof}

The function $R(z)$ has interesting properties even for arbitrary polynomial $p(z)$ but for self-interlacing
polynomials this function plays the most important role.

The following theorem will give us a tool to prove an analogue of the famous Hurwitz stability criterion for self-interlacing polynomials.

%
%
\begin{theorem}\label{Theorem.main.self-interlacing}
Let $p$ be a real polynomial of even (odd) degree $n$ as in~\eqref{main.polynomial}.
The polynomial $p\in\mathbf{SI}_I$ is self-interlacing if and only if its associated function~$\Phi$ defined
in~\eqref{assoc.function} maps the upper half-plane to itself and has exactly $l$ poles all of which are
positive (nonnegative) poles. Here $l$ is defined in~\eqref{floor.poly.degree}.
%
%
%
%
%
\end{theorem}
\begin{proof}
Indeed, let $p\in\mathbf{SI}_I$, so $p(z)$ and $p(-z)$ have no common roots. Then by Theorem~\ref{Theorem.SI.first.new} the function~$R(z)$ defined in~\eqref{auxiliary.rational.function} maps the upper half-plane to the lower half-plane.
At the same time, the function $\dfrac{1-w}{1+w}$ maps the lower half-plane to the upper half-plane as it is easy to check. Thus, from~\eqref{poly1.1} we obtain that the function $z\Phi(z^2)$ maps the upper half-plane to itself. Therefore (see e.g.~\cite[Theorem~3.4]{Holtz_Tyaglov}), $z\Phi(z^2)$ has the form
\begin{equation}\label{Theorem.main.self-interlacing.proof.1.new}
z\Phi(z^2)=\sum_{k=1}^r\dfrac{\alpha_k}{\mu_k-z}-\sum_{k=1}^r\dfrac{\alpha_k}{\mu_k+z}-\dfrac{\alpha_0}{z},
\end{equation}
where\footnote{It is easy to see that $l=r$ if $n$ is even, and $l=r+1$ if $n$ is odd.} $\alpha_0\geqslant0$ ($=0$ if and only if $n=2r$), $\alpha_k>0$, $k=1,\ldots,r$,
\begin{equation}\label{floor.poly.degree.2}
r\stackrel{def}{=}\left[\dfrac{n}2\right],
\end{equation}
and
\begin{equation*}
0<\mu_1<\mu_2<\cdots<\mu_r.
\end{equation*}
Here we take into account the facts that $z\Phi(z^2)$ is an odd function, and that the degree of the numerator of $z\Phi(z^2)$ is $n=\deg p$, while the degree of its denominator is $n-1$.
Thus the function $\Phi(u)$ has the form
\begin{equation}\label{Mittag.Leffler}
\Phi(u)=-\dfrac{\beta_0}{u}+\sum_{k=1}^r\frac{\beta_k}{\omega_k-u},
\end{equation}
where $\beta_0=\alpha_0\geqslant0$ ($=0$ if and only if $n=2r$), $\beta_k=2\alpha_k>0$, and $\omega_k=\mu_k^2>0$, $k=1,\ldots,r$,
Here $r$ is defined in~\eqref{floor.poly.degree.2}. So (see~\cite[Theorem 3.4]{Holtz_Tyaglov}) the
function $\Phi(u)$ maps the upper half-plane to itself, and all its poles are simple and positive for $n=2r$ or nonnegative for $n=2r+1$, as required.

Conversely, if $\Phi(u)$ maps the upper half-plane to itself, and has only positive poles whenever $n=2r$ and nonnegative poles whenever $n=2r+1$, then by Theorem 3.4 from~\cite{Holtz_Tyaglov}, the function
$\Phi(u)$ has the form~\eqref{Mittag.Leffler} with positive $\beta_k$, distinct positive $\omega_k$, $k=1,\ldots,r$, and nonnegative $\beta_0$ (which is zero only for even $n$).
Here we took into account that the degree of its numerator is less than the degree of its
denominator as it follows from~\eqref{assoc.function}. Thus, $z\Phi(z^2)$ can be presented as
in~\eqref{Theorem.main.self-interlacing.proof.1.new}, where $\lambda_k=\sqrt{\omega_k}>0$, and $\alpha_k=\dfrac{\beta_k}2>0$,
$k=1,\ldots,r$, $\alpha_0=\beta_0\geqslant0$ ($=0$ if and only if $n=2r$). Therefore~\cite[Chapter 3]{Holtz_Tyaglov}, the function $z\Phi(z^2)$ maps the upper half-plane to itself,
so by~\eqref{poly1.3} the function $R(z)$ maps the upper half-plane to the lower half-plane and has exactly $n$ poles. Now
Theorem~\ref{Theorem.SI.first.new} implies $p\in\mathbf{SI}_I$.
\end{proof}

\begin{remark}
Another, more complicated, proof of Theorem~\ref{Theorem.main.self-interlacing} can be found in the technical report~\cite{Tyaglov_GHP} (see Theorem~4.3 there).
\end{remark}

Note that the degrees of the numerator and denominator of $\Phi(u)$ are, respectively, $l-1$ and $l$, where $l$ is defined in~\eqref{floor.poly.degree}, so $\Phi(u)$ tends
to zero as $u$ tends to infinity. Expand now the function $\Phi(u)$ into Laurent series at infinity:
\begin{equation}\label{app.assoc.function.series}
\Phi(u)=\dfrac{p_1(u)}{p_0(u)}=\frac{s_0}u+\frac{s_1}{u^2}+\frac{s_2}{u^3}+\frac{s_3}{u^4}+\dots,
\end{equation}
and construct two sequences of Hankel determinants:
\begin{equation}\label{Hankel.determinants.1}
D_j(\Phi)=
\begin{vmatrix}
    s_0 &s_1 &s_2 &\dots &s_{j-1}\\
    s_1 &s_2 &s_3 &\dots &s_j\\
    \vdots&\vdots&\vdots&\ddots&\vdots\\
    s_{j-1} &s_j &s_{j+1} &\dots &s_{2j-2}
\end{vmatrix},\quad j=1,2,3,\dots,
\end{equation}
and
\begin{equation}\label{Hankel.determinants.2}
\widehat{D}_j(\Phi)=
\begin{vmatrix}
    s_1 &s_2 &s_3 &\dots &s_{j}\\
    s_2 &s_3 &s_4 &\dots &s_{j+1}\\
    \vdots&\vdots&\vdots&\ddots&\vdots\\
    s_{j} & s_{j+1} & s_{j+2} & \dots &s_{2j-1}
\end{vmatrix},\quad j=1,2,3,\dots,
\end{equation}
using the coefficients of the expansion~\eqref{app.assoc.function.series}.

From the results of~\cite[Chapter 3]{Holtz_Tyaglov} (see references there) one can obtain the following lemmata.
\begin{lemma}\label{Lemma.R-functions.Hankel.1}
The function $\Phi(u)$ maps the upper half-plane of the complex plane to itself and has exactly~$l$ poles
if and only if the following inequalities hold
\begin{equation*}\label{Theorem.total.nonnetativity.Hankel.matrix.conditions.1}
(-1)^jD_j(\Phi)>0,\qquad j=1,\ldots,l,
\end{equation*}
\begin{equation*}\label{Theorem.total.nonnetativity.Hankel.matrix.conditions.1.1}
D_j(\Phi)=0,\qquad j>l.
\end{equation*}
\end{lemma}
\begin{lemma}\label{Lemma.R-functions.Hankel.2}
If the function $\Phi(u)$ maps the upper half-plane of the complex plane to itself, then
it has only positive poles if and only if the following inequalities hold
\begin{equation}\label{Theorem.total.nonnetativity.Hankel.matrix.conditions.2}
(-1)^j\widehat{D}_j(\Phi)>0,\qquad j=1,\ldots,l.
\end{equation}
\begin{equation*}\label{Theorem.total.nonnetativity.Hankel.matrix.conditions.2.1}
\widehat{D}_j(\Phi)=0,\qquad j>l.
\end{equation*}
where $l$ is the number of poles of the function $\Phi(u)$.
\end{lemma}
\begin{lemma}\label{Lemma.R-functions.Hankel.3}
If the function $\Phi(u)$ maps the upper half-plane of the complex plane to itself, then
it has only nonnegative poles if and only if the following inequalities hold
\begin{equation}\label{Theorem.total.nonnetativity.Hankel.matrix.conditions.3}
(-1)^j\widehat{D}_j(\Phi)>0,\qquad j=1,\ldots,l-1.
\end{equation}
\begin{equation*}\label{Theorem.total.nonnetativity.Hankel.matrix.conditions.3.1}
\widehat{D}_j(\Phi)=0,\qquad j\geqslant l.
\end{equation*}
where $l$ is the number of poles of the function $\Phi(u)$.
\end{lemma}

Applying the famous Hurwitz formula (see \cite[p.\;214]{Gantmakher}, \cite[Theorem 1.5]{Holtz_Tyaglov} and references there) for determinants to the function~$\Phi(u)$, we get the following
\begin{equation}\label{Hurwitz.formula.1}
D_j(\Phi)=\dfrac1{a_0^{2j}}
\begin{vmatrix}
a_0&a_2&a_4&a_6&a_{8}&\dots&a_{4j-2}\\
0  &a_1&a_3&a_5&a_7&\dots&a_{4j-3}\\
0  &a_0&a_2&a_4&a_6&\dots&a_{4j-4}\\
0  &0  &a_1&a_3&a_5&\dots&a_{4j-5}\\
0  &0  &a_0&a_2&a_4&\dots&a_{4j-6}\\
\vdots&\vdots&\vdots&\vdots&\vdots&\ddots&\vdots\\
0  &0  &0  &0  &0  &\dots&a_{2j-1}
\end{vmatrix}=\dfrac1{a_0^{2j-1}}\Delta_{2j-1}(p),\quad j=1,\ldots,l,
\end{equation}
where $l$ is defined in~\eqref{floor.poly.degree}, and $\Delta_j(p)$ are defined in~\eqref{Hurwitz.minors}. Furthermore, noting that $\widehat{D}_j(\Phi)=D_j(u\Phi(u))$, $j=1,2,\ldots$ we obtain
\begin{equation}\label{Hurwitz.formula.2}
\widehat{D}_j(\Phi)=\dfrac1{a_0^{2j}}
\begin{vmatrix}
a_0&a_2&a_4&a_6&\dots&a_{4j-4}&a_{4j-2}\\
a_1&a_3&a_5&a_7&\dots&a_{4j-3}&a_{4j-1}\\
0  &a_0&a_2&a_4&\dots&a_{4j-6}&a_{4j-4}\\
0  &a_1&a_3&a_5&\dots&a_{4j-5}&a_{4j-3}\\
0  &0  &a_0&a_2&\dots&a_{4j-8}&a_{4j-6}\\
\vdots&\vdots&\vdots&\vdots&\ddots&\vdots&\vdots\\
0  &0  &0  &0  &\dots&a_{2j-2}  &a_{2j}\\
0  &0  &0  &0  &\dots&a_{2j-1}  &a_{2j+1}
\end{vmatrix}=\dfrac{(-1)^j}{a_0^{2j}}\Delta_{2j}(p),\quad j=1,\ldots,r,
\end{equation}
where $r$ is defined in~\eqref{floor.poly.degree.2}.

Thus, from Theorem~\ref{Theorem.main.self-interlacing}, Lemmata~\ref{Lemma.R-functions.Hankel.1}--\ref{Lemma.R-functions.Hankel.3} and
formul\ae~\eqref{Hurwitz.formula.1}--\eqref{Hurwitz.formula.2}, we obtain the following criterion of
self-interlacing, which is an analogue of the~Hurwitz stability
criterion.

\vspace{3mm}

%
\noindent\textbf{Theorem 1.2.} \textit{A real polynomial $p$ of degree $n$ as
in~\eqref{main.polynomial} belongs to the class $\mathbf{SI}_I$ if and only if
the Hurwitz minors $\Delta_j(p)$ satisfy the
inequalities:
\begin{equation}\label{Hurvitz.det.noneq.self-interlacing.2}
(-1)^j\Delta_{2j-1}(p)>0,\qquad j=1,\ldots,l,
\end{equation}
\begin{equation}\label{Hurvitz.det.noneq.self-interlacing.1}
\Delta_{2j}(p)>0,\qquad j=1,\ldots,r,
\end{equation}
where $l$ and $r$ are defined in~\eqref{floor.poly.degree} and~\eqref{floor.poly.degree.2}, respectively.}
%
\vspace{3mm}
\begin{proof}
Let $n=2l$. Then by Theorem~\ref{Theorem.main.self-interlacing}, $p(z)\in\mathbf{SI}_I$ if and only if the function $\Phi(u)$ defined in~\eqref{assoc.function}
maps the upper half-plane to the lower half-plane and has exactly $l$ positive poles ($l=r$ in this case). This is equivalent to the
inequalities~\eqref{Hurvitz.det.noneq.self-interlacing.2}--\eqref{Hurvitz.det.noneq.self-interlacing.1}, according to
Lemmata~\ref{Lemma.R-functions.Hankel.1}--\ref{Lemma.R-functions.Hankel.2} and formul\ae~\eqref{Hurwitz.formula.1}--\eqref{Hurwitz.formula.2}.

If $n=2l+1$, then according to Theorem~\ref{Theorem.main.self-interlacing}, $p(z)\in\mathbf{SI}_I$ if and only if the function $\Phi(u)$
maps the upper half-plane to the lower half-plane and has exactly $r$ ($r=l-1$ in this case) positive poles and one pole at zero ($l$ poles in total). By
Lemmata~\ref{Lemma.R-functions.Hankel.1} and~\ref{Lemma.R-functions.Hankel.3} and by formul\ae~\eqref{Hurwitz.formula.1}--\eqref{Hurwitz.formula.2}, this is equivalent to the
inequalities~\eqref{Hurvitz.det.noneq.self-interlacing.2}--\eqref{Hurvitz.det.noneq.self-interlacing.1}, as required.
\end{proof}

Note that~\eqref{Hurvitz.det.noneq.self-interlacing.2}
is equivalent to the following inequalities
\begin{equation*}\label{Hurvitz.det.noneq.self-interlacing.5}
\Delta_{2i-1}(p)\Delta_{2i+1}(p)<0,\quad i=0,1,\ldots,\left[\dfrac{n-1}2\right],
\end{equation*}
where $\Delta_{-1}(p)\equiv1$.

\vspace{3mm}

Following~\cite[Chapter 3]{Holtz_Tyaglov} we note that if the function $\Phi(u)$ maps the upper half-plane to itself, then we can simplify
the conditions~\eqref{Theorem.total.nonnetativity.Hankel.matrix.conditions.2} and~\eqref{Theorem.total.nonnetativity.Hankel.matrix.conditions.3}
using the coefficients of the denominator of $\Phi(u)$. Indeed, if a polynomial
$$
q(z)=d_0z^m+d_1z^{m-1}+\cdots+d_{m-1}z+d_m,\quad d_0>0,
$$
has only real roots, then by the Descartes Rule of Signs (see e.g.~\cite[Part V, Chapter 1]{Polya&Szego}) all its roots are positive if and only if
$$
d_{k-1}d_k<0,\qquad k=1,\ldots,m,
$$
or, that is the same,
$$
(-1)^kd_k<0,\qquad k=0,1,\ldots,m.
$$

This remark together with Theorem~\ref{Theorem.main.self-interlacing} and Lemma~\ref{Lemma.R-functions.Hankel.1} immediately implies
Theorem~\ref{Theorem.Lienard.Chipart.intro} which is an analogue of the stability criterion due to Li\'enard and Chipart~\cite{LienardChipart},~\cite[p.\;221]{Gantmakher} (see also~\cite{Holtz_Tyaglov} and references there).

%
%
%
%
%
%
Analogously to Theorem 11 from~\cite[Chap.\;XV, Sec.\;13]{Gantmakher} and Theorem 3.34 from~\cite{Holtz_Tyaglov}, one can easily establish three additional similar criteria for a real polynomial to be self-interlacing. We leave these simple exercises to the reader.

We end this section with the following remark.

\begin{remark}
A polynomial $p(z)$ of degree $n$ can always be represented as follows
\begin{equation*}
p(z)=\widehat{p}_0(z)+\widehat{p}_1(z),
\end{equation*}
where
\begin{equation}\label{poly.p_0}
\widehat{p}_0(z)=\dfrac{p(z)+(-1)^np(-z)}{2}=a_0z^{n}+a_2z^{n-2}+a_4z^{n-4}+\cdots,
\end{equation}
and
\begin{equation}\label{poly.p_1}
\widehat{p}_1(z)=\dfrac{p(z)-(-1)^np(-z)}{2}=a_1z^{n-1}+a_3z^{n-3}+a_5z^{n-5}+\cdots
\end{equation}

If $\widehat{p}_0(z)$ and $\widehat{p}_1(z)$ have real interlacing zeroes, then the
polynomial $p(z)=\widehat{p}_0(z)+\widehat{p}_1(z)$ has real zeroes as a linear combination of
polynomials with real interlacing zeroes~\cite{ChebotarevMeiman} (see also~\cite{Obreschkoff}). However, this notice
does not help to investigate the self-interlacing property of
polynomials. At the same time, this fact shows that roots of the self-interlacing polynomial
$p(z)$ interlace both roots of $\widehat{p}_0(z)$ and roots of $\widehat{p}_1(z)$, while the interlacing of roots of
$\widehat{p}_0(z)$ and $\widehat{p}_1(z)$, in this case, follows from the property of the function
$z\Phi(z^2)$ that maps the upper half-plane to itself~\cite[Chapter 3]{Holtz_Tyaglov}.
\end{remark}

\setcounter{equation}{0}

\section{Interrelation between Hurwitz stable and self-interlacing
polynomials}\label{section:connection.Hurwitz.SI}

In this section we establish that the classes of self-interlacing polynomials of kind $I$, $\mathbf{SI}_I$,
and \emph{real} Hurwitz stable polynomials (the polynomials whose roots lie
in the open left half-plane) are isomorphic by proving Theorem~\ref{Theorem.isomorphism}. This actually means that the set
of all stable real polynomials isomorphically embedded into the set of all
real polynomials with real simple roots. We prove the following theorem which is a rephrasing of Theorem~\ref{Theorem.isomorphism}.

\begin{theorem}\label{Theorem.connection.Hurwitz.self-interlacing}
A real polynomial
\begin{equation}\label{main.polynomial.2}
p(z)=\sum_{k=0}^na_kz^{n-k}
\end{equation}
belongs to the class $\mathbf{SI}_I$ if
 and only if the polynomial
\begin{equation}\label{main.polynomial.2.stable}
q(z)=\sum_{k=0}^n(-1)^{\tfrac{k(k+1)}2}a_kz^{n-k}
\end{equation}
is Hurwitz stable.

\end{theorem}

\begin{definition}
We call the polynomial $p\in\mathbf{SI}_I$ defined in~\eqref{main.polynomial.2} and the Hurwitz stable polynomial $q$ defined in~\eqref{main.polynomial.2.stable}
\emph{dual} to each other.
\end{definition}
\begin{proof}[Proof of Theorem~\ref{Theorem.connection.Hurwitz.self-interlacing}]
Indeed, by Theorem~\ref{Theorem.main.self-interlacing}, the polynomial $p\in\mathbf{SI}_I$ if and only if the function~\eqref{assoc.function}
\begin{equation*}\label{Theorem.connection.Hurwitz.self-interlacing.proof.1}
\Phi(u)=\dfrac{a_1u^{l-1}+a_3u^{l-2}+a_5u^{l-2}+\cdots}{a_0u^{l}+a_2u^{l-1}+a_4u^{l-2}+\cdots},\qquad l=\left[\dfrac{n+1}2\right],
\end{equation*}
maps the upper half-plane to itself and has only positive (for $n=2r$) or nonnegative (for $n=2r+1$) poles. In its turn, it is equivalent
to the fact that the function
\begin{equation*}\label{Theorem.connection.Hurwitz.self-interlacing.proof.2}
\Psi(u)=\Phi(-u)=\dfrac{-a_1u^{l-1}+a_3u^{l-2}-a_5u^{l-2}+\cdots}{a_0u^{l}-a_2u^{l-1}+a_4u^{l-2}+\cdots}=\sum_{j=0}^{+\infty}\dfrac{t_j}{z^{j+1}}
\end{equation*}
maps the upper half-plane to the lower half-plane and has only negative (for $n=2r$) or nonpositive (for $n=2r+1$) poles. Moreover,
\begin{equation*}
z\Psi(z^2)=\dfrac{q(z)-(-1)^nq(-z)}{q(z)+(-1)^nq(-z)},
\end{equation*}
so $\Psi(u)$ is the function associated with the polynomial $q$. From the results
of~\cite[Chapter 3]{Holtz_Tyaglov} (see also references there) it follows that $\Psi(u)$ maps the upper half-plane to the lower one if and only if
the Hankel determinants $D_k(\Psi)$ constructed with the coefficients $t_j$ of the Laurent series of $\Psi$ at infinity
satisfy the inequalities
\begin{equation*}\label{Theorem.connection.Hurwitz.self-interlacing.proof.3}
D_j(\Psi)>0,\quad j=1,\ldots,l,
\end{equation*}
so from the formula~\eqref{Hurwitz.formula.1} applied to the function $\Psi(u)$ and the polynomial $q$ we obtain
\begin{equation}\label{Theorem.connection.Hurwitz.self-interlacing.proof.4}
\Delta_{2j-1}(q)>0,\quad j=1,\ldots,l.
\end{equation}
Moreover, by Lemma~\ref{Lemma.R-functions.Hankel.1} the function $\Psi(u)$ maps the upper half-plane to itself if and only if
the following inequalities hold
\begin{equation*}\label{Theorem.connection.Hurwitz.self-interlacing.proof.5}
(-1)^j\widehat{D}_j(\Psi)>0,\quad j=1,\ldots,r,
\end{equation*}
so from~\eqref{Hurwitz.formula.2} we have
\begin{equation}\label{Theorem.connection.Hurwitz.self-interlacing.proof.6}
\Delta_{2j}(q)>0,\quad j=1,\ldots,r.
\end{equation}
The inequalities~\eqref{Theorem.connection.Hurwitz.self-interlacing.proof.4} and~\eqref{Theorem.connection.Hurwitz.self-interlacing.proof.6} mean that
all the Hurwitz determinants $\Delta(q)$ of the polynomial $q$ are positive, so $q$ is Hurwitz stable by Hurwitz
stability criterion~\cite{Gantmakher}.
\end{proof}

\begin{remark}
According to a formula in the proof of~\cite[Corollary~3.12]{Holtz_Tyaglov}, $|D_j(\Phi)|=|D_j(\Psi)|$, $j=1,\ldots,l$,
and $|\widehat{D}_j(\Phi)|=|\widehat{D}_j(\Psi)|$, $j=1,\ldots,r$. Thus, from the formul\ae~\eqref{Hurwitz.formula.1}--\eqref{Hurwitz.formula.2}, we
have
$$
|\Delta_i(p)|=|\Delta_i(q)|,\qquad i=1,\ldots,n,
$$
so the Hurwitz minors of the dual polynomials $p$ and $q$ differ only by signs. In Section~\ref{section:properties}
we prove that \textit{all} the Hurwitz minors of the polynomials $p$ and $q$ possess the same property.
\end{remark}

As an immediate consequence of
Theorem~\ref{Theorem.connection.Hurwitz.self-interlacing}, we
obtain the following analogue of Stodola's theorem claiming that all the coefficients of a real stable
polynomial are of the same sign~\cite{Gantmakher}.

\begin{theorem}\label{Th.Stodola.necessary.condition.self-interlacing}
If the polynomial $p$ defined in~\eqref{main.polynomial.2} belongs to the class $\mathbf{SI}_I$, then

\begin{equation}\label{Th.Stodola.necessary.condition.self-interlacing.condition}
(-1)^{\tfrac{j(j+1)}2}a_j>0,\qquad j=0,1,\ldots,n.
\end{equation}
\end{theorem}
\begin{proof}
In fact, if the polynomial $p$ is self-interlacing, then by Theorem~\ref{Theorem.connection.Hurwitz.self-interlacing} the
polynomial $q$ defined in~\eqref{main.polynomial.2.stable} is Hurwitz stable. But by Stodola's theorem its
coefficients are positive (since $a_0>0$), that implies~\eqref{Th.Stodola.necessary.condition.self-interlacing.condition}.
\end{proof}

\begin{remark}
Theorem~\ref{Th.Stodola.necessary.condition.self-interlacing} was
proved in~\cite{Fisk} by another method (see~\cite[Lemma 2.6]{Fisk}).
\end{remark}

Let us point out at one more interesting connection between
Hurwitz stable and self-interlacing polynomials. To do this we need the following simple fact.

\begin{prepos}\label{Proposition.stable.SI.connection}
If $p\in\mathbf{SI}_I$, then the Hurwitz stable polynomial $q$
defined in~\eqref{main.polynomial.2.stable}, can be represented as follows
\begin{equation}\label{self-in.stable.connection.one.more}
q(z)=i^{-n}p(iz)\dfrac{1-i}2+i^np(-iz)\dfrac{1+i}2.
\end{equation}
\end{prepos}
\begin{proof}
Let $n=2r$. Then the polynomial $p(z)$ can be represented a sum of two polynomials
\begin{equation*}
p(z)=p_0(z^2)+zp_1(z^2),
\end{equation*}
where
\begin{equation}\label{even.odd.parts.even.degree}
\begin{array}{l}
p_0(u)=a_0u^r+a_2u^{r-1}+\cdots+a_{2r},\\
\\
p_1(u)=a_1u^{r-1}+a_3u^{r-2}+\cdots+a_{2r-1}.\\
\end{array}
\end{equation}
So the corresponding dual Hurwitz stable polynomial $q(z)$ has the form
\begin{equation}\label{new.Hurwitz.poly}
q(z)=(-1)^r[p_0(-z^2)+zp_1(-z^2)].
\end{equation}

On the other hand, we have
\begin{equation*}
\begin{array}{l}
i^{-n}p(iz)=(-1)^rp(iz)=(-1)^r[p_0(-z^2)+izp_1(-z^2)],\\
\\
(-i)^{-n}p(-iz)=i^{n}p(-iz)=(-1)^rp(-iz)=(-1)^r[p_0(-z^2)-izp_1(-z^2)],
\end{array}
\end{equation*}
that implies
\begin{equation}\label{even.odd.parts}
\begin{array}{l}
(-1)^rp_0(-z^2)=\dfrac{i^{-n}p(iz)+i^np(-iz)}2,\\
\\
(-1)^rzp_1(-z^2)=-i\ \dfrac{i^{-n}p(iz)-i^np(-iz)}{2}.
\end{array}
\end{equation}
The formula~\eqref{self-in.stable.connection.one.more} now follows from~~\eqref{new.Hurwitz.poly} and~\eqref{even.odd.parts}.

The case of $n=2r+1$ can be proved analogously. The only difference we should take into account is that $p_0(u)=a_1u^r+a_3u^{r-1}+\cdots$ and
$p_1(u)=a_0u^r+a_2u^{r-1}+\cdots$, so $q(z)=(-1)^{r+1}[p_0(-z^2)-zp_1(-z^2)]$.
\end{proof}

Clearly, the converse formula is also true
\begin{equation*}
p(z)=i^{-n}q(iz)\dfrac{1-i}2+i^nq(-iz)\dfrac{1+i}2.
\end{equation*}
due to the duality.

Using Theorem~\ref{Theorem.connection.Hurwitz.self-interlacing} and Proposition~\ref{Proposition.stable.SI.connection} one can establish the following curious fact which was noticed by Yu.\,Barkovsky~\cite{Bark_private}.
\begin{theorem}\label{Theorem.SI_connection.second}
Let $p\in\mathbf{SI}_I$ and let $q$ be its dual Hurwitz stable polynomial. Then
\begin{equation*}
p(\lambda)=0\Longleftrightarrow\arg
q(i\lambda)=(-1)^{n-1}\dfrac\pi4\,\,\,\text{or}\,\,\,(-1)^{n-1}\dfrac{5\pi}4;
\end{equation*}
and, respectively, for $\mu\in\mathbb{R}$,
\begin{equation*}
q(\mu)=0\Longleftrightarrow\arg
p(i\mu)=(-1)^{n-1}\dfrac\pi4\,\,\,\text{or}\,\,\,(-1)^{n-1}\dfrac{5\pi}4.
\end{equation*}
In other word, if $p(\lambda)=0$, then $\Re q(i\lambda)=\Im q(i\lambda)$ or $\Re q(i\lambda)=-\Im q(i\lambda)$, and
if $q(\mu)=0$, $\mu\in\mathbb{R}$, then $\Re p(i\mu)=\Im p(i\mu)$ or $\Re p(i\mu)=-\Im p(i\mu)$.
\end{theorem}
\begin{proof}
Let first the degree of $p$ be even: $n=2r$. Then $p(\lambda)=0$ if and only if $\dfrac{\lambda p_1(\lambda^2)}{p_0(\lambda^2)}=-1$,
where $p_0$ and $p_1$ are defined in~\eqref{even.odd.parts.even.degree}. From~\eqref{new.Hurwitz.poly} we obtain
$q(i\lambda)=(-1)^r[p_0(\lambda^2)+i\lambda p_1(\lambda^2)]$.
%
%
Consequently, $\arg q(i\lambda)=\arctan\left(\dfrac{\lambda
p_1(\lambda^2)}{p_0(\lambda^2)}\right)=\arctan(-1)=-\dfrac\pi4\,\text{or}\,\dfrac{3\pi}4$.

The case $n=2r+1$ can be proved analogously with the difference that $\arg q(i\lambda)=\arctan(1)$ (see the proof of Proposition~\ref{Proposition.stable.SI.connection}). The second assertion of the theorem follows from the first one, since the polynomials $p(z)$ and $q(z)$ are dual.
\end{proof}
\noindent Theorem~\ref{Theorem.SI_connection.second} can be generalized for arbitrary polynomials, see Appendix for details.


Finally, let us note that from the proofs of Theorem~\ref{Theorem.connection.Hurwitz.self-interlacing} and Proposition~\ref{Proposition.stable.SI.connection}
one can define a self-interlacing polynomial in terms of roots of its odd and even parts. So the following fact is true.
\begin{prepos}
A real polynomial $p(z)=p_0(z^2)+zp_1(z^2)$ is self-interlacing if and only if the polynomials $p_0(u)$ and $p_1(u)$ have positive roots,
and the roots of one polynomial interlace the roots of the other one.
\end{prepos}
%

\setcounter{equation}{0}
\section{Properties of self-interlacing polynomials}\label{section:properties}


Hurwitz stable polynomials have an interpretation in terms of Stieltjes
continued fractions~\cite[Chapter~XV, \S 14, Theorem 16]{Gantmakher}.
Due to the relation between Hurwitz stable polynomials and self-interlacing
polynomials provided by Theorem~\ref{Theorem.connection.Hurwitz.self-interlacing},
it is possible to associate with self-interlacing polynomials certain continued
fractions of Stieltjes type considered e.g. in~\cite[Section 3.4]{Holtz_Tyaglov}.

The following theorem presents a relation between self-interlacing
polynomials and continued fractions of Stieltjes type.
\begin{theorem}\label{Theorem.self-interlacing.Stieltjes.cont.frac.criteria}
The polynomial $p$ of degree $n$ belongs to the class $\mathbf{SI}_I$ if and only if its associated
function~$\Phi(u)$ defined in~\eqref{assoc.function} has the following Stieltjes continued fraction
expansion:
\begin{equation}\label{Stieltjes.fraction.for.self-interlacing}
\Phi(u)=\dfrac1{c_1u+\cfrac1{c_2+\cfrac1{c_{3}u+\cfrac1{\ddots+\cfrac1{c_{2r-1}u+\cfrac1{c_{2r}+\cfrac1{c_{2r+1}u}}}}}}},\quad\text{with}\quad
(-1)^{i}c_i>0,\quad i=1,\ldots,2r,
\end{equation}
where $c_{2r+1}=\infty$ if $n$ is even, and $c_{2r+1}<0$ if $n$ is odd. The number $r$ is
as in~\eqref{floor.poly.degree.2}.
\end{theorem}
\begin{proof}
In fact, by Theorem~\ref{Theorem.main.self-interlacing}, $p\in\mathbf{SI}_I$ if and only if the
function $-\Phi (u)$ maps the upper half-plane to the lower half-plane. Now the~assertion of the
theorem follows from Theorem~3.8 and Corollaries~3.39 and~3.40 of the work~\cite{Holtz_Tyaglov}.
\end{proof}

According to~\cite{Holtz_Tyaglov}, the coefficients $c_i$ can be found by the following formul\ae
\begin{equation}\label{cont.frac.coeff}
c_i=\dfrac{\Delta_{i-1}^2(p)}{\Delta_{i-2}(p)\Delta_i(p)},\qquad i=1,\cdots,n.
\end{equation}
This formul\ae\ follow from~\eqref{Hurwitz.formula.1}--\eqref{Hurwitz.formula.2} and from formul\ae~(1.113)--(1.114) of the work~\cite{Holtz_Tyaglov}. The signs of $c_i$ in Theorem~\ref{Theorem.self-interlacing.Stieltjes.cont.frac.criteria}
follow from~\eqref{cont.frac.coeff}.

Using Theorem~\ref{Theorem.main.self-interlacing} one can easily obtain the following fact.
\begin{theorem}\label{Theorem.SI.with.differentiation}
Let $p\in\mathbf{SI}_I$ and $\deg p\geqslant2$.
Then the polynomial
\begin{equation*}\label{Theorem.SI.with.differentiation.condition}
p_j(z)=\sum\limits_{i=0}^{n-2j}\left[\dfrac{n-i}2\right]\left(\left[\dfrac{n-i}2\right]-1\right)\cdots\left(\left[\dfrac{n-i}2\right]+j-1\right)a_iz^{n-2j-i},\quad
j=1,\ldots,r-1.
\end{equation*}
belongs to the class $\mathbf{SI}_I$. Here $r$ is defined in~\eqref{floor.poly.degree.2}.
\end{theorem}
\begin{proof}
Let $n=2r$. By Theorems~\ref{Theorem.main.self-interlacing}, if $p\in\mathbf{SI}_I$, then the associated function $-\Phi(u)=-\dfrac{p_1(u)}{p_0(u)}$, where
$q_0(u)$ and $q_1(u)$ are defined in~\eqref{even.odd.parts.even.degree}, maps the upper half-plane to the lower half-plane, and has only positive poles. This is equivalent
to the fact (see~\cite[Theorem~3.4]{Holtz_Tyaglov}) that $q_0(u)$ and $q_1(u)$ have simple, negative, and interlacing roots, and $-\Phi(u)$ is decreasing
between its poles.

By V.A.\,Markov theorem~\cite[Chapter~1, Theorem~9]{ChebotarevMeiman} (see also~\cite[Theorem~3.8]{Holtz_Tyaglov}) if two polynomials have real, simple, and interlacing roots, then their derivatives
also have real, simple, and interlacing roots. Thus, for every $j=1,\ldots,r-1$, the roots of the derivatives of order $j$ of the polynomials $q_0(u)$ and $q_1(u)$
also have simple, positive, and interlacing roots. Moreover, since the leading coefficients of $q_0^{(j)}(u)$ and $q_1^{(j)}(u)$ are always of different sings,
the functions $-\dfrac{q_1^{(j)}(u)}{q_0^{(j)}(u)}$ are negative for sufficiently large positive $u$. Therefore, by Theorem~\ref{Theorem.main.self-interlacing}
the polynomials $p_j(z)=q_0^{(j)}(z^2)+zq_1^{(j)}(z^2)$ belong to the class $\mathbf{SI}_I$ for all $j=1,\ldots,r-1$.

The case $n=2r+1$ can be proved analogously.
\end{proof}

Using V.A.\,Markov's theorem it is also easy to prove the following fact.
\begin{theorem}
If $p(z)\in\mathbf{SI}_I$, then $p^{(k)}(z)\in\mathbf{SI}_I$, $k=1,\ldots,n-1$, where $p^{(k)}(z)$ is the $k^{\mathrm{th}}$ derivative of~$p(z)$.
\end{theorem}
\begin{proof}
The self-interlacing polynomials $p(z)$ and $p(-z)$ have real, simple, and interlacing roots by definition. For any $k=1,\ldots,n-1$, the polynomials $p^{(k)}(z)$ and $p^{(k)}(-z)$ also have real, simple, and interlacing roots according to V.A.\,Markov's theorem. Moreover,
the largest root of $p^{(k)}(z)$ is greater than the largest root of $p^{(k)}(-z)$ (see~\cite[Corollary~3.7]{Holtz_Tyaglov} for detailed proof
of this fact), so $p^{(k)}(z)\in\mathbf{SI}_I$ for
all $k=1,\ldots,n-1$.
\end{proof}
%


Now we are in a position to study minors of the Hurwitz matrix of self-interlacing polynomials. Recall that
the Hurwitz matrix of a polynomial $p(z)$ defined as in~\eqref{main.polynomial} has the form
\begin{equation}\label{HurwitzMatrix}
\mathcal{H}_n(p)=
\begin{pmatrix}
a_1&a_3&a_5&a_7&\dots&0&0\\
a_0&a_2&a_4&a_6&\dots&0&0\\
0  &a_1&a_3&a_5&\dots&0&0\\
0  &a_0&a_2&a_4&\dots&0&0\\
\vdots&\vdots&\vdots&\vdots&\ddots&\vdots&\vdots\\
0  &0  &0  &0  &\dots&a_{n-1} &0\\
0  &0  &0  &0  &\dots&a_{n-2} &a_n
\end{pmatrix}.
\end{equation}
The Hurwitz minors defined in~\eqref{Hurwitz.minors} are the leading principal minors of $\mathcal{H}_n(p)$.

Let $p\in\mathbf{SI}_I$, and let $q$ be its dual Hurwitz stable polynomial. From the proof of Proposition~\ref{Proposition.stable.SI.connection} it follows that if $p(z)=p_0(z^2)+zp_1(z^2)$, where $p_0(u)$ and $p_1(u)$ are the odd end even parts of $p$, then
$q(z)=(-1)^{l}[p_0(-z^2)+(-1)^nzp_1(-z^2)]$, where $l$ is defined in~\eqref{floor.poly.degree}. The
Hurwitz matrix of the polynomial~$q$ has the form
\begin{equation*}
\mathcal{H}_n(q)=
\begin{pmatrix}
-a_1& a_3&-a_5& a_7&\dots&0\\
 a_0&-a_2& a_4&-a_6&\dots&0\\
0  &-a_1& a_3&-a_5&\dots&0\\
0  & a_0&-a_2& a_4&\dots&0\\
\dots&\dots&\dots&\dots&\dots&\dots\\
0&0&0&0&\dots&(-1)^{\tfrac{n(n+1)}2}a_n\\
\end{pmatrix}.
\end{equation*}
It is easy to see that the matrix $\mathcal{H}_n(q)$ can be
factorized as follows
\begin{equation}\label{SI_Stab.connection.7}
\mathcal{H}_n(q)=C_n\mathcal{H}_n(p)E_n,
\end{equation}
where the $n\times n$ matrices $C_n$ and
$E_n$ have the forms
\begin{equation*}
C_n=
\begin{pmatrix}
    1 &0  &0  & 0 & 0 &\dots\\
    0 & -1 &0  & 0 & 0 &\dots\\
    0 &0  &-1 & 0 & 0 &\dots\\
    0 &0  &0  &1 & 0 &\dots\\
    0 &0  &0  & 0 & 1 &\dots\\
    \vdots&\vdots&\vdots&\vdots&\vdots&\ddots
\end{pmatrix},
\qquad
E_n=
\begin{pmatrix}
    -1 &0  &0  & 0 & 0 &\dots\\
    0 & 1 &0  & 0 & 0 &\dots\\
    0 &0  &-1 & 0 & 0 &\dots\\
    0 &0  &0  &1 & 0 &\dots\\
    0 &0  &0  & 0 & -1 &\dots\\
    \vdots&\vdots&\vdots&\vdots&\vdots&\ddots
\end{pmatrix}.
\end{equation*}
All the non-principal minors of these matrices  equal zero. The principal minors of these
matrices can be easily calculated:
\begin{equation}\label{Matrix.Unique.1.minors}
C_n
\begin{pmatrix}
    i_1 &i_2 &\dots &i_m\\
    i_1 &i_2 &\dots &i_m
\end{pmatrix}=
(-1)^{\sum\limits_{k=1}^m\tfrac{i_k(i_k-1)}2},
\qquad
E_n\begin{pmatrix}
    i_1 &i_2 &\dots &i_m\\
    i_1 &i_2 &\dots &i_m\\
\end{pmatrix}
=(-1)^{\sum\limits_{k=1}^{m}i_k},
\end{equation}
where $1\leqslant i_1<i_2<\ldots<i_m\leqslant n$. Thus, the
Cauchy--Binet formula together with~\eqref{SI_Stab.connection.7}
and~\eqref{Matrix.Unique.1.minors} implies

\begin{equation}\label{SI_Stab.connection.6}
\mathcal{H}_n(q)
\begin{pmatrix}
    i_1 &i_2 &\dots &i_m\\
    j_1 &j_2 &\dots &j_m
\end{pmatrix}=(-1)^{\sum\limits_{k=1}^m\tfrac{i_k(i_k-1)}2+\sum\limits_{k=1}^mj_k}
\mathcal{H}_n(p)
\begin{pmatrix}
    i_1 &i_2 &\dots &i_m\\
    j_1 &j_2 &\dots &j_m
\end{pmatrix},
\end{equation}
where $1\leqslant
\begin{array}{c}
i_1<i_2<\ldots<i_m\\
j_1<j_2<\ldots<j_m
\end{array}
\leqslant n$.

\vspace{2mm}

Thus, we established that the absolute values of the corresponding minors of the Hurwitz matrices of the
dual polynomials $p$ and $q$ are equal as we announced in Section~\ref{section:connection.Hurwitz.SI}.

\begin{remark}
Note that if we know in advance formul\ae~\eqref{main.poly}--\eqref{main.poly.q} (see also~\eqref{main.polynomial.2}--\eqref{main.polynomial.2.stable}),
then the formula~\ref{SI_Stab.connection.7} implies Theorems~\ref{Theorem.self-interlacing.Hurwitz.criterion}--\ref{Theorem.Lienard.Chipart.intro}.
However, our proof of Theorem~\ref{Theorem.connection.Hurwitz.self-interlacing} is based on some properties of $R$-functions that immediately
imply Theorems~\ref{Theorem.self-interlacing.Hurwitz.criterion}--\ref{Theorem.Lienard.Chipart.intro} without addressing to Hurwitz matrices.
\end{remark}
\begin{remark}
It is clear that the formula~\eqref{SI_Stab.connection.6} is true for two arbitrary complex polynomials $p$ and $q$ related as in~\eqref{main.polynomial.2}--\eqref{main.polynomial.2.stable}.
\end{remark}

The polynomial $q$ is stable by assumption. Consequently, by Asner's theorem~\cite{Asner} (see also~\cite{Holtz_st,Holtz_Tyaglov,Dyachenko,Kemperman}), the matrix $\mathcal{H}_n(q)$
is totally nonnegative, that is, any its minor is nonnegative. From this fact and the formula~\eqref{SI_Stab.connection.6} we obtain the following theorem.
\begin{theorem}
Let a polynomial $p$ be defined in~\eqref{main.polynomial.2}, and let $\mathcal{H}_n(p)$ be its Hurwitz matrix defined
in~\eqref{HurwitzMatrix}. If $p\in\mathbf{SI}_I$, then
%
%
%
%
%
\begin{equation*}
(-1)^{\sum\limits_{k=1}^m\tfrac{i_k(i_k-1)}2+\sum\limits_{k=1}^mj_k}
\mathcal{H}_n(p)
\begin{pmatrix}
    i_1 &i_2 &\dots &i_m\\
    j_1 &j_2 &\dots &j_m
\end{pmatrix}\geqslant0,
\end{equation*}
%
%
where  $1\leqslant
\begin{array}{c}
i_1<i_2<\ldots<i_m\\
j_1<j_2<\ldots<j_m
\end{array}
\leqslant n$. Moreover, the absolute values of the corresponding minors of the matrices $\mathcal{H}_n(p)$ and $\mathcal{H}_n(q)$, where $q$ is the polynomials
dual to $p$ as in~\eqref{main.polynomial.2.stable}, are equal.
\end{theorem}

\setcounter{equation}{0}

\section{The second proof of the Hurwitz self-interlacing
criterion. Stability criterion.}\label{section:Second.proof.of.SI.criterion}

In this section, we provide another approach to proving Theorem~\ref{Theorem.self-interlacing.Hurwitz.criterion}.
This approach is based on Hankel minors related to the Laurent series at infinity of the function $R(z)$ defined in~\eqref{auxiliary.rational.function} instead of
the more standard function~$\Phi(u)$ defined in~\eqref{assoc.function}.

Let again
\begin{equation}\label{SI.second.proof.poly}
p(z)=a_0z^n+a_1z^{n-1}+\dots+a_n,\qquad a_1,\dots,a_n\in\mathbb
R,\ a_0>0,
\end{equation}
be a real polynomial. Consider the function $R(z)$ defined in~\eqref{auxiliary.rational.function}:
\begin{equation}\label{Rat.func.2}
R(z)=\dfrac{(-1)^np(-z)}{p(z)},
\end{equation}
and expand it into its Laurent series at $\infty$:
\begin{equation*}
R(z)=1+\frac{s_0}z+\frac{s_1}{z^2}+\frac{s_2}{z^3}+\dots,
\end{equation*}

According to Kronecker's theorem (see~\cite[p.~426]{Holtz_Tyaglov} and references there), rank of the
matrix $S=\|s_{j+k}\|_0^{\infty}$ is equal to the number of poles of the function $R$. It is clear that rank of $S$
equals~$n$ if the polynomials $p(z)$ and $p(-z)$ have no common
zeroes.

%
\begin{lemma}\label{lem.Determinants.relations.1}
For the function $R$ in~\eqref{Rat.func.2}, the
following formul\ae~hold:
\begin{equation}\label{Determinants.relations.1}
a_0^{2j}D_j(R)=(-1)^{\frac{j(j+1)}2}2^ja_0\Delta_{j-1}(p)\Delta_j(p),\quad
j=1,2,\ldots
\end{equation}
where $\Delta_j(p)$ are defined in~\eqref{Hurwitz.minors}, $\Delta_0(p)\equiv1$.
\end{lemma}
\begin{proof}
By the Hurwitz formula (see \cite[p.\;214]{Gantmakher}, \cite[Theorem 1.5]{Holtz_Tyaglov} and references there) applied to the function~$R(z)$, we have
\begin{equation*}
a_0^{2j}D_j(R)=
\begin{vmatrix}
    a_0 &a_1 &a_2 &a_3 &\dots &a_{j-1} & a_{j}  &\dots &a_{2j-2} & a_{2j-1}\\
    a_0 &-a_1&a_2 &-a_3&\dots &-a_{j-1}& a_{j}  &\dots &a_{2j-2} &-a_{2j-1}\\
     0  &a_0 &a_1 &a_2 &\dots &a_{j-2} & a_{j-1}&\dots &a_{2j-3} &a_{2j-2}\\
     0  &a_0 &-a_1&a_2 &\dots &a_{j-2} &-a_{j-1}&\dots &-a_{2j-3}&a_{2j-2}\\
     0  & 0  &a_0 &a_1 &\dots &a_{j-3} & a_{j-2}&\dots &a_{2j-4} & a_{2j-3}\\
     0  & 0  &a_0 &-a_1&\dots &-a_{j-3}& a_{j-2}&\dots &a_{2j-4} &-a_{2j-3}\\
    \dots&\dots&\dots&\dots&\dots&\dots&\dots&\dots&\dots&\dots\\
     0  &  0 &  0 &  0 &\dots &a_{0} & a_{1}&\dots &a_{j-1} &a_{j}\\
     0  &  0 &  0 &  0 &\dots &a_{0} &-a_{1}&\dots &-a_{j-1}&a_{j}\\
\end{vmatrix}=
\end{equation*}
\begin{equation*}
=2^j\begin{vmatrix}
    a_0 &a_1 &a_2 &a_3 &\dots &a_{j-1} & a_{j}  &\dots &a_{2j-2} & a_{2j-1}\\
    a_0 &  0 &a_2 &  0 &\dots &   0    & a_{j}  &\dots &a_{2j-2} &    0    \\
     0  &a_0 &a_1 &a_2 &\dots &a_{j-2} & a_{j-1}&\dots &a_{2j-3} &a_{2j-2}\\
     0  &a_0 &  0 &  0 &\dots &a_{j-2} &    0   &\dots &    0    &a_{2j-2}\\
     0  & 0  &a_0 &a_1 &\dots &a_{j-3} & a_{j-2}&\dots &a_{2j-4} & a_{2j-3}\\
     0  & 0  &a_0 &  0 &\dots &    0   & a_{j-2}&\dots &a_{2j-4} &    0    \\
    \dots&\dots&\dots&\dots&\dots&\dots&\dots&\dots&\dots&\dots\\
     0  &  0 &  0 &  0 &\dots &a_{0} & a_{1}&\dots &a_{j-1} &a_{j}\\
     0  &  0 &  0 &  0 &\dots &a_{0} &   0  &\dots &    0   &a_{j}\\
\end{vmatrix}=
\end{equation*}
\begin{equation*}
=(-2)^j\begin{vmatrix}
    a_0 &  0 &a_2 &  0 &\dots &   0    & a_{j}  &\dots &a_{2j-2} &    0    \\
     0  &a_1 &  0 &a_3 &\dots &a_{j-1} &   0    &\dots &    0    & a_{2j-1}\\
     0  &a_0 &  0 &a_2 &\dots &a_{j-2} &   0    &\dots &    0    &a_{2j-2}\\
     0  & 0  &a_1 &  0 &\dots &   0    & a_{j-1}&\dots &a_{2j-3} &    0    \\
     0  & 0  &a_0 &  0 &\dots &   0    & a_{j-2}&\dots &a_{2j-4} &    0    \\
     0  & 0  &  0 &a_1 &\dots &a_{j-3} &   0    &\dots &    0    &a_{2j-3}\\
    \dots&\dots&\dots&\dots&\dots&\dots&\dots&\dots&\dots&\dots\\
     0  &  0 &  0 &  0 &\dots &a_{0} &   0  &\dots &    0   &a_{j}\\
     0  &  0 &  0 &  0 &\dots &   0  & a_{1}&\dots &a_{j-1} &  0  \\
\end{vmatrix}=
\end{equation*}
\begin{equation*}
=(-2)^ja_0\begin{vmatrix}
    a_1 &  0 &a_3 &  0  &\dots &a_{j-1} &   0    &\dots &    0    & a_{2j-1}\\
    a_0 &  0 &a_2 &  0  &\dots &a_{j-2} &   0    &\dots &    0    &a_{2j-2}\\
      0 &  0 &a_1 &  0  &\dots &a_{j-3} &   0    &\dots &    0    &a_{2j-3}\\
      0 &  0 &a_0 &  0  &\dots &a_{j-4} &   0    &\dots &    0    &a_{2j-4}\\
    \dots&\dots&\dots&\dots&\dots&\dots&\dots&\dots&\dots&\dots\\
      0 &  0 &  0 &  0  &\dots &a_{1}   &   0    &\dots &    0    &a_{j+1}\\
      0 &  0 &  0 &  0  &\dots &a_{0}   &   0    &\dots &    0    &a_{j}\\
      0 &a_1 &  0 &a_3  &\dots &   0    & a_{j-1}&\dots &a_{2j-3} &    0    \\
      0 &a_0 &  0 &a_2  &\dots &   0    & a_{j-2}&\dots &a_{2j-4} &    0    \\
    \dots&\dots&\dots&\dots&\dots&\dots&\dots&\dots&\dots&\dots\\
      0 &  0 &  0 &  0  &\dots &   0    & a_{2}  &\dots &a_{j}   &  0  \\
      0 &  0 &  0 &  0  &\dots &   0    & a_{1}  &\dots &a_{j-1} &  0  \\
\end{vmatrix}=
\end{equation*}
\begin{equation*}
=(-2)^ja_0(-1)^{\frac{j(j-1)}2}\begin{vmatrix}
    a_1 &a_3 &a_5 &\dots &a_{2j-1} &   0    &   0    &\dots &    0    &   0\\
    a_0 &a_2 &a_4 &\dots &a_{2j-2} &   0    &   0    &\dots &    0    &   0\\
      0 &a_1 &a_3 &\dots &a_{2j-3} &   0    &   0    &\dots &    0    &   0\\
      0 &a_0 &a_2 &\dots &a_{2j-4} &   0    &   0    &\dots &    0    &   0\\
    \dots&\dots&\dots&\dots&\dots&\dots&\dots&\dots&\dots&\dots\\
      0 &  0 &  0 &\dots &a_{j+1} &   0    &   0    &\dots &    0    &   0\\
      0 &  0 &  0 &\dots &a_{j}   &   0    &   0    &\dots &    0    &   0\\
      0 &  0 &  0 &\dots &   0    & a_1    & a_{3}  &\dots &a_{2j-5} &a_{2j-3}\\
      0 &  0 &  0 &\dots &   0    & a_0    & a_{2}  &\dots &a_{2j-6} &a_{2j-4}\\
    \dots&\dots&\dots&\dots&\dots&\dots&\dots&\dots&\dots&\dots\\
      0 &  0 &  0 &\dots &  0     &   0    &    0   &\dots &a_{j-2} &a_{j}   \\
      0 &  0 &  0 &\dots &  0     &   0    &    0   &\dots &a_{j-3} &a_{j-1}\\
\end{vmatrix}=
\end{equation*}
\begin{equation*}
=(-1)^{\frac{j(j+1)}2}2^ja_0\Delta_{j-1}(p)\Delta_j(p).
\end{equation*}
Here we set $a_j=0$ for $j>n$.
\end{proof}

\begin{remark}
Note that the formul\ae~\eqref{Determinants.relations.1}
can be obtained (overcoming certain difficulties) from some
theorems of the book~\cite{Wall}. But it is simpler to deduce
them directly as we have done above.
\end{remark}

Using this lemma it is easy to prove the equivalence of the
conditions $1)$ and $2)$ of
Theorem~\ref{Theorem.self-interlacing.Hurwitz.criterion}.

\begin{proof}[The second proof of Theorem~\ref{Theorem.self-interlacing.Hurwitz.criterion}]
According to Theorem~\ref{Theorem.SI.first.new}, $p\in\mathbf{SI}_I$ if and only if
the function $R(z)$ maps the upper half-plane of the complex plane to the lower half-plane, and has
exactly $n$ poles. Consequently (see e.g.~\cite[Theorem~3.4]{Holtz_Tyaglov} and references there),
the following inequalities hold
\begin{equation*}
D_j(R)>0,\qquad j=1,\ldots,n,
\end{equation*}
so from~\eqref{Determinants.relations.1} we obtain
\begin{equation}\label{Second.proof.relations}
(-1)^{\frac{j(j+1)}2}\Delta_{j-1}(p)\Delta_j(p)>0,\qquad
j=1,\ldots,n.
\end{equation}

Multiplying the inequalities~\eqref{Second.proof.relations} for
$j=2m$ and $j=2m-1$, we obtain
\begin{equation*}
\Delta_{2m-1}^2\Delta_{2m}\Delta_{2m-2}>0.
\end{equation*}
Consequently, the minors $\Delta_{2i}(p)$, $i=1,\ldots,r$, are
positive, since $\Delta_0(p)=1$, so the
inequalities~\eqref{Hurvitz.det.noneq.self-interlacing.1} hold.

If we multiply the inequalities~\eqref{Second.proof.relations} for
$j=2m$ and $j=2m+1$, we get
\begin{equation*}
-\Delta_{2m}^2(p)\Delta_{2m-1}(p)\Delta_{2m+1}(p)>0.
\end{equation*}
These inequalities
imply~\eqref{Hurvitz.det.noneq.self-interlacing.2}.

The converse assertion can be proved in the same way. That is, the
inequalities~\eqref{Hurvitz.det.noneq.self-interlacing.2}--\eqref{Hurvitz.det.noneq.self-interlacing.1}
imply the inequalities~\eqref{Second.proof.relations} which, in
turn, imply the inequalities $D_j(R)>0$, $j=1,\ldots,n$, according
to~\eqref{Determinants.relations.1}. According to
~\cite[Theorem~3.4]{Holtz_Tyaglov}, the function
$R(z)$ maps the upper half-plane of the complex plane to the lower half-plane, and has exactly $n$ poles, so $p\in\mathbf{SI}_I$
by Theorem~\ref{Theorem.SI.first.new}, as required.
\end{proof}

From Lemma~\ref{lem.Determinants.relations.1} and from Hurwitz's
stability criterion claiming the positivity of the Hurwitz minors~$\Delta_j(p)$, $j=1,\ldots,n$,
for any Hurwitz stable polynomial of degree $n$ and vice versa  (\cite{Hurwitz,Gantmakher}, see
also~\cite{Krein_Naimark} and references there) we get the following stability criterion.

\begin{theorem}\label{Theorem.new.stability.criterion}
For a real polynomial $p$ of degree $n$, the following statements are equivalent:
\begin{itemize}
\item[(1)] The polynomial $p$ is Hurwitz stable.
\item[(2)] The function $R$ defined in~\eqref{Rat.func.2}
maps the open right half-plane of the complex into the open unit disc, and has exactly $n$ poles.
\item[(3)] For the function $R$, the following
inequalities hold
\begin{equation*}
(-1)^{\frac{j(j+1)}2}D_j(R)>0,\quad j=1,2,\ldots,n.
\end{equation*}
\end{itemize}
\end{theorem}
\begin{proof}
The equivalence $(1)$ and $(3)$ follows from the Hurwitz stability criterion and Lemma~\ref{lem.Determinants.relations.1}
as we mentioned above.

Let $p(z)$ be Hurwitz stable. Then $p(z)$ and $p(-z)$ have no common roots,
so $R(z)$ has exactly $n$ poles. The polynomial $p(z)$ can be represented as follows
$$
p(z)=a_0\prod\limits_{k}(z-\lambda_k)\prod\limits_j(z-\mu_j)(z-\overline{\mu}_j),
$$
where $\lambda_k<0$, $\Re\mu_j<0$.

Consequently, the function $R(z)$ has the form
$$
R(z)=\prod\limits_k\dfrac{z+\lambda_k}{z-\lambda_k}\cdot
\prod\limits_j\dfrac{z+\mu_j}{z-\overline{\mu}_j}\cdot\dfrac{z+\overline{\mu}_j}{z-\mu_j}.
$$
It is clear now that for any $z$ such that $\Re z>0$ we have $|R(z)|<1$, since
$$
\left|\dfrac{z+a}{z-\overline{a}}\right|<1,
$$
for any $z$ and $a$ such that $\Re z>0$ and $\Re a<0$, and $R(z)$ is a product of such functions. Additionally, it is easy to see that $R(z)$ maps the
imaginary axis into the unit circle and the open left half-plane into the exterior of the closed unit disc.

Conversely, if $R(z)$ maps the open right half-plane into the unit circle, and has exactly $n$ poles, then it has no poles in the open right half-plane.
It also has no poles on the imaginary axis, since any pure imaginary zero of $p(z)$ is a zero of $p(-z)$, but
they have no common zeroes by assumption. So all poles of $R(z)$ (the zeroes of $p(z)$) lie in the open left half-plane,
as required.
\end{proof}
%


\section*{Acknowledgement}
The author is Shanghai Oriental Scholar whose work was supported by Russian Science Foundation, grant no. 14-11-00022.

\setcounter{equation}{0}

\section{Appendix}

Let $p$ be a real polynomial as in~\eqref{main.polynomial.2}
with $a_n\neq0$, and
let $q$ be defined as in~\eqref{main.polynomial.2.stable}. In the same way as used in the proof of Theorem~\ref{Theorem.SI_connection.second}, it is easy
to establish that if $\lambda\in\mathbb{R}$, then
\begin{equation*}
p(\lambda)=0\Longleftrightarrow\arg
q(i\lambda)=(-1)^{n-1}\dfrac\pi4\,\,\,\text{or}\,\,\,(-1)^{n-1}\dfrac{5\pi}4.
\end{equation*}

This fact is related to the following formula
\begin{equation}\label{tan.arctan.formula}
\tan\sum_{k=1}^{n}\arctan(a_k)=\dfrac{\sum\limits_{k=1}^{n}a_k-
\sum\limits_{i=1}^{l-1}e_{2i+1}(a_1,\ldots,a_n)}
{1-\sum\limits_{i=1}^{r}e_{2i}(a_1,\ldots,a_n)}
\end{equation}
where $l$ and $r$ are defined in~\eqref{floor.poly.degree} and in~\eqref{floor.poly.degree.2}, respectively, and
$$
e_k(x_1,\ldots,x_n)=\sum\limits_{1\leqslant i_1<\cdots<i_{k}\leqslant n}x_{i_1}x_{i_2}\cdots x_{i_{k}},
$$
is the $k$th symmetric function.

Formula~\eqref{tan.arctan.formula} can be proved by induction from the well-known formula (see e.g.~\cite{Gradshtein.Ryzhik}):
$$
\tan(\arctan(x)+\arctan(y))=\dfrac{x+y}{1-xy}.
$$

If now we represent the polynomial $p$ as follows
$$
p(z)=a_0\prod\limits_{k=1}^n(z-\lambda_k),
$$
then for $\lambda\in\mathbb{R}$ such that $\arg p(i\lambda)=(-1)^{n-1}\dfrac{\pi}4$ or $(-1)^{n-1}\dfrac{5\pi}4$, we have, by~\eqref{tan.arctan.formula},
$$
\tan\left(\arg p(i\lambda)\right)=\tan\sum_{k=1}^{n}\arctan\left(\dfrac{\Im\lambda_k-\lambda}{\Re\lambda_k}\right)=\dfrac{\lambda p_1(-\lambda^2)}{p_0(-\lambda^2)}=(-1)^{n-1},
$$
Now from the formula $q(z)=(-1)^{l}[p_0(-z^2)+(-1)^nzp_1(-z^2)]$, it follows that $q(\lambda)=0$.

Also, formula~\eqref{tan.arctan.formula} can be used to find the roots of the dual polynomial $q$ if the roots of $p$ are known.
\begin{example}
Consider the polynomial
$$
p(z)=(z+a)^n=\sum\limits_{k=0}^n\binom{n}{k}z^ka^{n-k}
$$
which is stable for $a>0$. By Theorem~\ref{Theorem.SI_connection.second}, its dual polynomial
$$
q(z)=\sum\limits_{k=0}^n(-1)^{\tfrac{k(k+1)}{2}}\binom{n}{k}z^ka^{n-k}
$$
has $n$ distinct \textit{real} roots $\mu_k$, $k=1,\ldots,n$, satisfying the condition
$$
\arg p(i\mu_k)=n\arctan\left(\dfrac{\mu_k}a\right)=(-1)^{n-1}\left(\dfrac{\pi}{4}+\pi k\right).
$$
Therefore,
$$
\mu_k=(-1)^{n-1}a\cdot\tan\dfrac{\pi(4k+1)}{4n},\qquad k=1,\ldots,n.
$$
Note that if $a$ is an arbitrary non-zero complex number, the roots of the polynomial $q$
have the same form.

In a similar way, it is possible to find the roots of the self-interlacing polynomial $q(z)$ dual to a stable polynomial with only two
distinct roots of equal multiplicity. We leave such an exercise to the reader.

\end{example}

\addcontentsline{toc}{section}{Bibliography} 

\end{document}